\documentclass[12pt]{amsart}

\usepackage{mathrsfs,amssymb}

\newtheorem{theorem}{Theorem}[section]
\newtheorem{lemma}[theorem]{Lemma}

\theoremstyle{definition}
\newtheorem{definition}[theorem]{Definition}
\newtheorem{que}[theorem]{Question}

\theoremstyle{remark}
\newtheorem{remark}[theorem]{Remark}

\numberwithin{equation}{section}

\let \la=\lambda
\let \e=\varepsilon
\let \d=\delta
\let \o=\omega
\let \a=\alpha
\let \f=\varphi

\let \O=\Omega

\begin{document}
\title[Coifman-Fefferman and Fefferman-Stein inequalities]
{A characterization of the weighted weak type Coifman-Fefferman and Fefferman-Stein inequalities}

\author{Andrei K. Lerner}
\address{Department of Mathematics,
Bar-Ilan University, 5290002 Ramat Gan, Israel}
\email{lernera@math.biu.ac.il}

\thanks{The author was supported by ISF grant No. 447/16 and ERC Starting Grant No. 713927.}

\begin{abstract}
We introduce a variant of the $C_p$ condition (denoted by $SC_p$), and show that it characterizes weighted weak type versions of
the classical Coifman-Fefferman and Fefferman-Stein inequalities.
\end{abstract}


\maketitle

\section{Introduction}
This paper concerns two long-standing open problems of characterizing the weights $w$ that satisfy the Coifman-Fefferman \cite{CF} inequality
\begin{equation}\label{CF}
\|Tf\|_{L^p(w)}\le C\|Mf\|_{L^p(w)}
\end{equation}
and the Fefferman-Stein \cite{FS} inequality
\begin{equation}\label{FS}
\|f\|_{L^p(w)}\le C\|f^{\#}\|_{L^p(w)}.
\end{equation}
Here $T$ is a Calder\'on-Zygmund operator, and $M$ and $f^{\#}$ are the maximal and the sharp maximal operators, respectively.

Originally (\ref{CF}) was established in \cite{CF} for weights satisfying the $A_{\infty}$ condition
 (in fact, a good-$\lambda$ inequality relating $T$ and $M$ had been already obtained in the weighted setting with $w\in A_{\infty}$ in an earlier work of Coifman \cite{C} and (\ref{CF}) is implicit there). Inequality (\ref{FS}) was established in \cite{FS} in the unweighted setting
but the method in \cite{FS} is easily extended to $w\in A_{\infty}$.

Recall that one of the equivalent definitions of $A_{\infty}$ says that this is the class of weights satisfying the reverse H\"older inequality,
namely $w\in A_{\infty}$ if there exist $C>0$ and $r>1$ such that for every cube $Q$,
$$\left(\frac{1}{|Q|}\int_Qw^r\right)^{1/r}\le C\frac{1}{|Q|}\int_Qw.$$

In \cite{M}, Muckenhoupt showed that the $A_{\infty}$ condition is not necessary for (\ref{CF}); he also established that (\ref{CF}) for the Hilbert transform implies
the so-called $C_p$ condition which he conjectured to be sufficient for (\ref{CF}). Observe that this conjecture is still open.

In the $n$-dimensional case the $C_p$ condition can be formulated as follows: $w\in C_p$
if there exist $C>0$ and $r>1$ such that for every cube~$Q$,
$$\left(\frac{1}{|Q|}\int_Qw^r\right)^{1/r}\le C\frac{1}{|Q|}\int_{{\mathbb R}^n}(M\chi_Q)^pw.$$
It is easy to see that for every $q>p>0$,
$$A_{\infty}\subset C_q\subset C_p.$$

In \cite{S2}, Sawyer extended Muchenhoupt's result by showing that (\ref{CF}) for each of the Riesz transforms $R_j, j=1,\dots,n$ implies the $C_p$ condition;
also Sawyer gave a partial answer to Muchenhoupt's conjecture proving that the $C_{p+\e}$ condition for some $\e>0$ is sufficient for (\ref{CF}).

In \cite{Y}, Yabuta obtained an analogue of Sawyer's result for (\ref{FS}). Namely, he showed that the $C_p$ condition is necessary for (\ref{FS}) and
the $C_{p+\e}$ condition for some $\e>0$ is sufficient for (\ref{FS}). Thus a natural analogue of Muckenhoupt's conjecture for (\ref{FS}) is that the
$C_p$ condition is necessary and sufficient for (\ref{FS}).

In \cite{L2}, it was shown  that there is a condition $\widetilde C_p$ such that $C_{p+\e}\subset \widetilde C_p\subset C_{p}$ for
every $\e>0$ and $\widetilde C_p$ is sufficient for (\ref{FS}).

By the above results of Sawyer and Yabuta, the $C_p$ conjectures for both (\ref{CF}) and (\ref{FS}) would be easily solved if the following self-improving
property $C_{p}\Rightarrow C_{p+\e}$ was true. However, it was shown by
Kahanp\"a\"a and Mejlbro~\cite{KM} in the one-dimensional case that there exist $C_p$ weights that do not belong to $C_{p+\e}$ for every $\e>0$.
The Kahanp\"a\"a-Mejlbro construction has been recently extended to higher dimensions in the work by Canto, Li, Roncal and Tapiola \cite{CLRT}.

We also mention recent works \cite{Ca,CLPR} where different aspects of the $C_p$ theory have been investigated. In particular,
it was shown in \cite{CLPR} that  for $p>0$ the $C_{\max(1,p)+\e}$ condition is sufficient for
\begin{equation}\label{spcl}
\|A_{\mathcal S}f\|_{L^p(w)}\le C\|Mf\|_{L^p(w)},
\end{equation}
which provides a different approach to (\ref{CF}). Here $A_{\mathcal S}$ is the sparse operator defined by
$$A_{\mathcal S}f(x)=\sum_{Q\in {\mathcal S}}\left(\frac{1}{|Q|}\int_Q|f|\right)\chi_Q,$$
and ${\mathcal S}$ is a sparse family.

Observe that although both (\ref{CF}) and (\ref{FS}) are known to hold for all $p>0$,
 the above mentioned results in \cite{M,S2,Y,L2} are obtained in the case $p>1$.

The proofs of the necessity of the $C_p$ condition for (\ref{CF}) (with the Riesz transforms) in \cite{S2} and for (\ref{FS}) in \cite{Y} actually show that the $C_p, p>1,$ condition is also necessary for
the weak type estimates
\begin{equation}\label{CFW}
\|Tf\|_{L^{p,\infty}(w)}\le C\|Mf\|_{L^p(w)}
\end{equation}
and
\begin{equation}\label{FSW}
\|f\|_{L^{p,\infty}(w)}\le C\|f^{\#}\|_{L^p(w)}.
\end{equation}
However, even for these, weaker versions of (\ref{CF}) and (\ref{FS}) the sufficiency of the $C_p$ condition is an open question.

In this paper we characterize (\ref{FSW}) and a variant of (\ref{CFW}) by means of a condition which seems to be stronger than the $C_p$ condition.

\begin{definition}\label{strcp} Let $p>0$. We say that a weight $w$ satisfies the $SC_p$ (strong $C_p$) condition if
there exist $C>0$ and $r>1$ such that for every family of pairwise disjoint cubes $\{Q_j\}$,
$$\sum_j\left(\frac{1}{|Q_j|}\int_{Q_j}w^r\right)^{1/r}|Q_j|\le C\int_{{\mathbb R}^n}(M\chi_{\cup_jQ_j})^pw.$$
\end{definition}

A number of equivalent definitions of the $SC_p$ condition is given in Section 3. In the trivial case when the family $\{Q_j\}$
consists of  one cube only, we obtain the $C_p$ condition. So, obviously, $SC_p\subset C_p$. On the other hand, if $p>1$, then
$C_{p+\e}\subset SC_p$ (see Section 6 for further discussion on the relationship between $SC_p$ and $C_p$).

Given a Calder\'on-Zygmund operator $T$, define its maximal truncation $T^{\star}$ by
$$T^{\star}f(x)=\sup_{\e>0}|T(f\chi_{\{y:|y-x|>\e\}})(x)|.$$

\begin{theorem}\label{singchar} If $p>0$ and $w\in SC_p$, then for every Calder\'on-Zygmund operator $T$ with Dini-continuous kernel,
\begin{equation}\label{weaksing}
\|T^{\star}f\|_{L^{p,\infty}(w)}\le C\|Mf\|_{L^p(w)}.
\end{equation}
Conversely, if $p>1$ and (\ref{weaksing}) holds for each of the maximal truncated Riesz transforms $R^{\star}_j, j=1,\dots, n$, then $w\in SC_p$.
\end{theorem}

Observe that since $|Tf|\le |T^{\star}f|+c|f|$ (see \cite[p. 36]{St1}), Theorem~\ref{singchar} implies that the $SC_p$ condition is also sufficient for (\ref{CFW}). However, in our
proof of the necessity part of Theorem \ref{singchar}, the assumption that (\ref{weaksing}) holds for the maximal Riesz transforms
(and for $p>1$) is crucial. It is still not clear to us how to deduce the necessity
of the $SC_p$ condition even in the one-dimensional case assuming (\ref{CFW}) for the Hilbert transform (not maximally truncated).

It turns out that the necessity of the $SC_p$ condition for the weak Fefferman-Stein inequality (\ref{FSW}) is quite easy for every $p>0$, and the following theorem holds.

\begin{theorem}\label{FefS} Let $p>0$. The inequality (\ref{FSW}) holds
if and only if $w$ satisfies the $SC_p$ condition.
\end{theorem}

The sufficiency parts of both Theorems \ref{singchar} and \ref{FefS} are corollaries of the corresponding weak type
analogue of (\ref{spcl}).

\begin{theorem}\label{weaksp}
Let ${\mathscr D}$ be a dyadic lattice and let ${\mathcal S}\subset {\mathscr D}$ be an $\eta$-sparse family.
Let $p>0$ and assume that $w$ satisfies the $SC_p$ condition. Then
\begin{equation}\label{asm}
\|A_{{\mathcal S}}f\|_{L^{p,\infty}(w)}\le C\|Mf\|_{L^p(w)},
\end{equation}
where $C>0$ does not depend on $f$.
\end{theorem}

The proof of this theorem is based essentially on the technique
developed by Domingo-Salazar, Lacey and Rey \cite{DLR} in order to prove a weighted weak type $(1,1)$ estimate for $A_{\mathcal S}$
with an arbitrary weight.

Since (\ref{FSW}) is derived from (\ref{asm}) and, by Theorem \ref{FefS}, the $SC_p$ condition is necessary for (\ref{FSW}),
we obtain that the $SC_p$ condition is also necessary for (\ref{asm}), in general.

The paper is organized as follows. Section 2 contains necessary definitions and preliminary facts. In Section 3, we obtain several characterizations of the $SC_p$ condition. Section 4 is devoted to proving Theorem \ref{weaksp}. In Section 5 we prove Theorems \ref{singchar} and \ref{FefS}. Section 6 contains some concluding remarks and open questions.

\section{Preliminaries}
In this section we provide necessary definitions and facts that will be used in the rest of the paper.

\subsection{Dyadic lattices, sparse families, and Calder\'on-Zygmund operators}
Given a cube $Q_0\subset {\mathbb R}^n$, let ${\mathcal D}(Q_0)$ denote the set of all dyadic cubes with respect to $Q_0$, that is, the cubes
obtained by repeated subdivision of $Q_0$ and each of its descendants into $2^n$ congruent subcubes.

The following definition was given in \cite{LN}.
\begin{definition}\label{dl}
A dyadic lattice ${\mathscr D}$ in ${\mathbb R}^n$ is any collection of cubes such that
\begin{enumerate}
\renewcommand{\labelenumi}{(\roman{enumi})}
\item
if $Q\in{\mathscr D}$, then each child of $Q$ is in ${\mathscr D}$ as well;
\item
every 2 cubes $Q',Q''\in {\mathscr D}$ have a common ancestor, i.e., there exists $Q\in{\mathscr D}$ such that $Q',Q''\in {\mathcal D}(Q)$;
\item
for every compact set $K\subset {\mathbb R}^n$, there exists a cube $Q\in {\mathscr D}$ containing $K$.
\end{enumerate}
\end{definition}

Let ${\mathscr D}$ be a dyadic lattice. We say that a family ${\mathcal S}\subset {\mathscr D}$ is $\eta$-sparse, $0<\eta<1$, if for every cube $Q\in {\mathcal S}$,
$$\Big|\bigcup_{Q'\in {\mathcal S}: Q'\subsetneq Q}Q'\Big|\le (1-\eta)|Q|.$$
In particular, if ${\mathcal S}\subset {\mathscr D}$ is $\eta$-sparse, then defining for every $Q\in {\mathcal S}$,
$$E_Q=Q\setminus \bigcup_{Q'\in {\mathcal S}: Q'\subsetneq Q}Q',$$
we obtain that $|E_Q|\ge \eta |Q|$ and the sets $\{E_Q\}_{Q\in {\mathcal S}}$ are pairwise disjoint.

We say that $T$ is a Calder\'on-Zygmund operator with Dini-continuous kernel if $T$ is a linear operator of weak type $(1,1)$ such that
$$Tf(x)=\int_{{\mathbb R}^n}K(x,y)f(y)dy\quad\text{for all}\,\,x\not\in \text{supp}\,f$$
with kernel $K$ satisfying the smoothness condition
$$
|K(x,y)-K(x',y)|\le \o\left(\frac{|x-x'|}{|x-y|}\right)\frac{1}{|x-y|^n}
$$
for $|x-x'|<|x-y|/2$, where $\int_0^1\o(t)\frac{dt}{t}<\infty.$

We will use the following result. Its different versions and proofs can be found in \cite{CR1,HRT,La,L4,LN,LO}.

\begin{theorem}\label{point} Let $T$ be a Calder\'on-Zygmund operator with Dini-continuous kernel. Then for every compactly supported
$f\in L^1({\mathbb R}^n)$, there exist $3^n$ dyadic lattices ${\mathscr D}_j$ and $\eta_n$-sparse families ${\mathcal S}_j\subset {\mathscr D}_j$  such that
for a.e. $x\in {\mathbb R}^n$,
$$T^{\star}f(x)\le C_{n,T}\sum_{j=1}^{3^n}A_{{\mathcal S}_j}f(x).$$
\end{theorem}

\subsection{Maximal operators and $\la$-oscillations}
For a locally integrable function $f$, define the Hardy-Littlewood maximal function $Mf$ and the Fefferman-Stein sharp maximal function $f^{\#}$ by
$$Mf(x)=\sup_{Q\ni x}\frac{1}{|Q|}\int_Q|f|\quad\text{and}\quad f^{\#}(x)=\sup_{Q\ni x}\frac{1}{|Q|}\int_Q|f-f_Q|,$$
where the supremum is taken over all cubes $Q\subset {\mathbb R}^n$ containing the point $x$, and $f_Q=\frac{1}{|Q|}\int_Qf$.

The non-increasing rearrangement of a
measurable function $f$ on ${\mathbb R}^n$ is defined by
$$f^*(t)=\inf\Big\{\a>0:|\{x\in {\mathbb R}^n:|f(x)|>\a\}|\le
t\Big\}\quad(0<t<\infty).$$

Given a measurable function $f$, a cube $Q$ and $0<\la<1$, the $\la$-oscillation of $f$ over $Q$ is defined by
$$\o_{\la}(f;Q)=\inf_c\big((f-c)\chi_Q\big)^*\big(\la|Q|\big).$$

The local sharp maximal function $M_{\la}^{\#}f$ is defined by
$$M_{\la}^{\#}f(x)=\sup_{Q\ni x}\o_{\la}(f;Q)\quad(0<\la<1).$$
It is well known (see \cite{JT,L1}) that the  sharp function $f^{\#}$ can be viewed as the maximal operator
acting on $M_{\la}^{\#}f$, namely, for all $x\in {\mathbb R}^n$,
\begin{equation}\label{two}
c_1MM^{\#}_{\la}f(x)\le f^{\#}(x)\le c_2MM^{\#}_{\la}f(x) \quad(0<\la\le 1/2),
\end{equation}
where $c_1$ depends on $\la$ and $n$ and $c_2$ depends only on $n$.

In \cite{LN}, the notion of the $\la$-oscillation is defined a bit differently:
$$\tilde \o_{\la}(f;Q)=\inf\{\o(f;E):E\subset Q, |E|\ge (1-\la)|Q|\},$$
where
$$\o(f;E)=\sup_Ef-\inf_Ef.$$

It is easy to see that
\begin{equation}\label{oscr}
\tilde \o_{\la}(f;Q)\le 2\o_{\la}(f;Q)\quad(0<\la<1).
\end{equation}
Indeed, observe that for every constant $c$,
$$\o(f;E)=\o(f-c;E)\le 2\sup_E|f-c|.$$
Let
$$E=\{x\in Q:|f(x)-c|\le \big((f-c)\chi_Q\big)^*\big(\la|Q|\big)\}.$$
Then $|E|\ge (1-\la)|Q|$, and therefore,
$$\tilde\o_{\la}(f;Q)\le 2\big((f-c)\chi_Q\big)^*\big(\la|Q|\big),$$
which implies (\ref{oscr}).

Let $S_0({\mathbb R}^n)$ be the space of measurable functions $f$ on
${\mathbb R}^n$ such that for any $\a>0$,
$$|\{x\in
{\mathbb R}^n:|f(x)|>\a\}|<\infty.$$

In \cite{LN}, the following result was proved (for a local version of this result see \cite{H,L3}).

\begin{theorem}\label{md}
Let $f\in S_0({\mathbb R}^n)$. For every dyadic lattice ${\mathscr D}$, there exists a $\frac{1}{6}$-sparse family ${\mathcal S}\subset {\mathscr D}$ (depending on $f$)
such that
$$|f|\le \sum_{Q\in {\mathcal S}}\tilde \o_{2^{-n-2}}(f;Q)\chi_Q$$
almost everywhere.
\end{theorem}

\subsection{H\"older's inequality for $L\log L$}
Given a Young function $\Phi$ and a cube $Q$, define the normalized Orlicz average $\|f\|_{\Phi,Q}$ by
$$
\|f\|_{\Phi,Q}=\inf\Big\{\a>0:\frac{1}{|Q|}\int_Q\Phi(|f(y)|/\a)dy\le 1\Big\}.
$$

Denote $\|f\|_{L\log L,Q}$ if $\Phi(t)=t\log(e+t)$ and $\|f\|_{\text{exp}L,Q}$ if $\Phi(t)=e^t-1$.
Then the following generalized H\"older's inequality holds (see, e.g., \cite[p. 166]{W}):
$$
\frac{1}{|Q|}\int_Q|fg|dx\lesssim \|f\|_{L\log L,Q}\|g\|_{\text{exp}L,Q}.
$$

A simple computation shows that if $E\subset Q$, then
$$\|\chi_E\|_{\Phi,Q}=\frac{1}{\Phi^{-1}(|Q|/|E|)}.$$
This along with H\"older's inequality implies
\begin{equation}\label{H}
\int_E|f|\lesssim \frac{|Q|}{\log(1+|Q|/|E|)}\|f\|_{L\log L,Q}\quad(E\subset Q).
\end{equation}

\subsection{A reverse $L\log L$ estimate for the Riesz transforms}
A well known result of Stein \cite{St} says that
\begin{equation}\label{logl}
\|f\|_{L\log L, Q}|Q|\lesssim \int_QM(f\chi_Q)dx.
\end{equation}

In the same work \cite{St}, Stein mentioned (without a proof) that for
the standard Riesz transforms defined by
$$R_jf(x)=\lim_{\e\to 0} c_n\int_{|y|>\e}f(x-y)\frac{y_j}{|y|^{n+1}}dy\quad(j=1,\dots,n)$$
the following analogue of (\ref{logl}) holds: if $f\ge 0$ on a cube $\a Q, \a>1,$ and $R_jf\in L^1(\a Q)$ for
every $j=1,\dots,n$, then $f\in L\log L(Q)$.
We will need a quantitative version of this result, similar to (\ref{logl}). Probably the proof of the following
statement is well known but we could not find it in the literature, and therefore it is given below.

\begin{lemma}\label{str} For every cube $Q$ and a non-negative function $f$ on $Q$,
$$
\|f\|_{L\log L,Q}|Q|\lesssim \sum_{j=1}^n\int_{3Q}|R_j(f\chi_{Q})|dx+\int_Qfdx.
$$
\end{lemma}

\begin{proof} Define the Poisson maximal function
$$M(f,P)(x)=\sup_{t>0}|P_t*f(x)|,$$
where $P$ is the Poisson kernel. By an equivalent characterization of the Hardy space $H^1$ (see \cite[p. 141]{G}),
\begin{equation}\label{eqh}
\|M(f,P)\|_{L^1}\lesssim \sum_{j=1}^n\|R_jf\|_{L^1}+\|f\|_{L^1}.
\end{equation}

Since $M(f\chi_Q)\sim M(f\chi_Q,P)$, by (\ref{logl}) we obtain
\begin{eqnarray}
\|f\|_{L\log L,Q}|Q|&\lesssim& \int_QM(f\chi_Q,P)dx\label{intst}\\
&\lesssim& \int_QM\big((f-f_Q)\chi_Q,P\big)dx+\int_Qfdx.\nonumber
\end{eqnarray}

Further, by (\ref{eqh}),
\begin{equation}\label{aaa}
\int_QM\big((f-f_Q)\chi_Q,P\big)dx\lesssim \sum_{j=1}^n\|R_j\big((f-f_Q)\chi_{Q}\big)\|_{L^1}+\int_Qfdx.
\end{equation}

By the standard estimate for singular integrals (see, e.g., \cite[p. 231]{G}),
$$
\int_{{\mathbb R}^n\setminus 3Q}|R_j\big((f-f_Q)\chi_Q\big)|dx\lesssim \|f-f_Q\|_{L^1(Q)}\lesssim \int_Qfdx.
$$
Therefore, using also that $\|R_j(\chi_Q)\|_{L^1(3Q)}\lesssim |Q|$, we obtain
\begin{eqnarray*}
\sum_{j=1}^n\|R_j\big((f-f_Q)\chi_{Q}\big)\|_{L^1}&\lesssim& \sum_{j=1}^n\|R_j\big((f-f_Q)\chi_{Q}\big)\|_{L^1(3Q)}+\int_Qfdx\\
&\lesssim& \sum_{j=1}^n\|R_j(f\chi_Q)\|_{L^1(3Q)}+\int_Qfdx,
\end{eqnarray*}
which, combined with  (\ref{intst}) and (\ref{aaa}) , completes the proof.
\end{proof}

\section{Characterizations of the $SC_p$ condition}
In this section we obtain several equivalent definitions of the $SC_p$ condition.
An important role will be played by the following simple lemma.

\begin{lemma}\label{Mchi}
Let $f\in L^1_{loc}({\mathbb R}^n)$. For all $\a>0$ and for all $x\in {\mathbb R}^n$,
$$M\chi_{\{Mf>\a\}}(x)\le \frac{9^n}{\a}Mf(x).$$
\end{lemma}

\begin{proof}
By homogeneity, it suffices to prove that for every cube $Q$ containing $x$,
\begin{equation}\label{suf}
\frac{|Q\cap {\{Mf>1\}}|}{|Q|}\le 9^nMf(x).
\end{equation}

Let $y\in Q$, and let $Q'$ be an arbitrary cube containing $y$. Then either $Q'\subset 3Q$ or $Q\subset 3Q'$. Therefore,
\begin{equation}\label{estM}
Mf(y)=\sup_{Q'\ni y}\frac{1}{|Q'|}\int_{Q'}|f|\le \max\big(M(f\chi_{3Q})(y), 3^n\inf_QMf\big).
\end{equation}

If $3^n\inf_QMf>1$, then (\ref{suf}) holds trivially. If $3^n\inf_QMf\le 1$, then by (\ref{estM}) and by the weak type $(1,1)$ of $M$,
\begin{eqnarray*}
|\{y\in Q:Mf(y)>1\}|&\le& |\{y\in Q:M(f\chi_{3Q})(y)>1\}|\\
&\le& 3^n\int_{3Q}|f|.
\end{eqnarray*}
Therefore, in this case we again obtain (\ref{suf}).
\end{proof}

\begin{definition}\label{sep}
Let $R\ge 1$. We say that a family of cubes $\{Q_j\}$ is $R$-separated if the cubes $RQ_j$ are pairwise disjoint.
\end{definition}

In the following theorem, a set $E$ is assumed to be bounded measurable set of positive measure.

\begin{theorem}\label{eqv} Let $p>0$ and let $w$ be a weight. The following conditions are equivalent.
\begin{enumerate}
\renewcommand{\labelenumi}{(\roman{enumi})}
\item $w\in SC_p$.
\item
For every $R\ge 1$, there exists $C>0$ such that for every $R$-separated family of cubes $\{Q_j\}$,
$$\sum_j\|w\|_{L\log L,Q_j}|Q_j|\le C\int_{{\mathbb R}^n}(M\chi_{\cup_jQ_j})^pw.$$
\item There exists a continuous function $\f$ on $(0,1)$ with $\displaystyle\lim_{\la\to 0}\f(\la)=0$ such that for every set $E$ and for all $0<\la<1$,
$$w(E)\le \f(\la)\int_{{\mathbb R}^n}(M\chi_{\{M\chi_E>\la\}})^pw.$$
\item There exist $0<C,\la_0<1$ such that for every set $E$,
$$\int_{{\mathbb R}^n}(M\chi_E)^pw\le C\int_{{\mathbb R}^n}(M\chi_{\{M\chi_E>\la_0\}})^pw.$$
\item
There exist $C,\d>0$ such that for every set $E$ and for all $0<\la<1$,
$$w(E)\le C\la^{\d}\int_{{\mathbb R}^n}(M\chi_{\{M\chi_E>\la\}})^pw.$$
\end{enumerate}
\end{theorem}

\begin{proof}
The implication $(\rm{i})\Rightarrow (\rm{ii})$ is trivial since
$$\|w\|_{L\log L,Q}\le C_r\left(\frac{1}{|Q|}\int_Qw^r\right)^{1/r}$$
for every $r>1$.

Turn to the implication $(\rm{ii})\Rightarrow (\rm{iii})$. Let $E$ be a bounded set of positive measure, and let $0<\la<1$. Denote
$$\Omega=\{x:M\chi_E>\la\}.$$

Let $R\ge 1$ as in condition (ii). By the Whitney covering lemma (as stated in \cite{S2}), there is a covering $\O=\cup_jQ_j$
such that
\begin{equation}\label{cond}
C_1Q_j\cap \Omega^c\not=\emptyset\quad\text{and}\quad\sum_j\chi_{RQ_j}(x)\le C_2,
\end{equation}
where $C_1$ and $C_2$ depend only on $R$ and $n$.

The first condition in (\ref{cond}) implies that
\begin{equation}\label{inter}
|Q_j\cap E|\le C_1^n\la|Q_j|.
\end{equation}
In turn, the second condition in (\ref{cond}) implies that the family $F=\{Q_j\}$ can be written as the union of $N$ $R$-separated families $F_i$, where
$N$ depends only on $C_2$ and $n$ (see \cite[p. 69]{KK} for the proof of this fact).

Applying condition (ii) along with (\ref{H}) and (\ref{inter}), we obtain
\begin{eqnarray*}
\sum_{Q_j\in F_i}w(E\cap Q_j)&\lesssim& \frac{1}{\log(C/\la)}\sum_{Q_j\in F_i}\|w\|_{L\log L,Q_j}|Q_j|\\
&\lesssim& \frac{1}{\log(C/\la)}\int_{{\mathbb R}^n}(M\chi_{\{M\chi_E>\la\}})^pw.
\end{eqnarray*}
Therefore,
$$
w(E)=\sum_{i=1}^N\sum_{Q_j\in F_i}w(E\cap Q_j)\lesssim \frac{1}{\log(C/\la)}\int_{{\mathbb R}^n}(M\chi_{\{M\chi_E>\la\}})^pw,
$$
which proves (iii).

Let us show now that $(\rm{iii})\Rightarrow (\rm{iv})$.
Let $0<\tau<1$. By the Calder\'on-Zygmund decomposition, if $\frac{|Q\cap E|}{|Q|}\le \tau$, then
$$|Q\cap E|\le 2^n\tau|Q\cap \{M\chi_E>\tau\}|.$$
Therefore,
$$M\chi_E(x)\le \tau\Rightarrow M\chi_E(x)\le 2^n\tau M\chi_{\{M\chi_E>\tau\}}.$$
From this,
$$
\int_{{\mathbb R}^n}(M\chi_E)^pw\le (2^n\tau)^p\int_{{\mathbb R}^n}(M\chi_{\{M\chi_E>\tau\}})^pw+w(\{M\chi_E>\tau\}).
$$

Next, condition (iii) combined with Lemma \ref{Mchi} implies
$$w(\{M\chi_E>\tau\})\le \f(\la)\int_{{\mathbb R}^n}(M\chi_{\{M\chi_E>\la\tau/9^n\}})^pw.$$
Hence, taking $\tau=\tau'=2^{-n-2}$ and $\la=\la'$ such that $\f(\la')\le 1/4$, we obtain condition (iv) with $C=\frac{1}{2}$ and $\la_0=\la'\tau'/9^n$.

Turn to the proof of $(\rm{iv})\Rightarrow (\rm{v})$. Iterating (iv) along with Lemma~\ref{Mchi} yields
$$\int_{{\mathbb R}^n}(M\chi_E)^pw\le C^k\int_{{\mathbb R}^n}(M\chi_{\{M\chi_E>(\la_0/9^n)^k\}})^pw$$
for all $k\in {\mathbb N}$. From this,
$$w(E)\le (9^n/\la_0)^{\d}\la^{\d}\int_{{\mathbb R}^n}(M\chi_{\{M\chi_E>\la\}})^pw$$
for all $0<\la<1$, where $\d=\frac{\log C}{\log(\la_0/9^n)}$.

It remains to show that $({\rm{v}})\Rightarrow SC_p$. Let $\{Q_j\}$ be a family of pairwise disjoint cubes.
Take
$$E_j\subset \{x\in Q_j:w(x)\ge (w\chi_{Q_j})^*(\la|Q_j|)\}$$
with $|E_j|=\la|Q_j|$.
Denote $E=\cup_jE_j$.
Let us prove that
\begin{equation}\label{show}
w(E)\lesssim \big(\la^{\d}\chi_{(0,1/3^n)}(\la)+\chi_{[1/3^n,1)}(\la)\big)\int_{{\mathbb R}^n}(M\chi_{\cup_jQ_j})^pw.
\end{equation}

If $1/3^n\le \la<1$, then (\ref{show}) is trivial since $w(E)\le w(\cup_jQ_j)$.
Suppose that $0<\la<1/3^n$. Then we claim that
\begin{equation}\label{emb}
\{M\chi_E>3^n\la\}\subset \cup_j3Q_j.
\end{equation}
Indeed, let $\frac{|Q\cap E|}{|Q|}>3^n\la$. Denote by $F$ a subfamily of those $Q_j$ having non-empty intersection with $Q$.
If $Q_j\subset 3Q$ for all $Q_j\in F$, then
$$|Q\cap E|\le \sum_{E_j\subset Q_j\in F}|E_j|=\la\sum_{Q_j\in F}|Q_j|\le 3^n\la|Q|,$$
which is a contradiction. Therefore, there exists $Q_j\in F$ such that $Q\subset 3Q_j$, which proves (\ref{emb}).

Applying condition (v) along with (\ref{emb}) yields
\begin{eqnarray*}
w(E)&\le& C(3^n\la)^{\d}\int_{{\mathbb R}^n}(M\chi_{\{M\chi_E>3^n\la\}})^pw\\
&\le& C'\la^{\d}\int_{{\mathbb R}^n}(M\chi_{\cup_j3Q_j})^pw.
\end{eqnarray*}
Since $\cup_j3Q_j\subset \{M\chi_{\cup_jQ_j}\ge 1/3^n\}$, Lemma \ref{Mchi} along with the previous estimate completes the proof of (\ref{show}).

It follows from (\ref{show}) that
\begin{eqnarray*}
&&\sum_j|Q_j|(w\chi_{Q_j})^*(\la|Q_j|)\le \frac{1}{\la}w(E)\\
&&\lesssim \big(\la^{\d-1}\chi_{(0,1/3^n)}(\la)+\la^{-1}\chi_{[1/3^n,1)}(\la)\big)\int_{{\mathbb R}^n}(M\chi_{\cup_jQ_j})^pw.
\end{eqnarray*}
From this, rewriting the standard estimate
$$\|w\|_{L^r(Q_j)}\le \|w\|_{L^{r,1}(Q_j)}$$
as
$$\left(\frac{1}{|Q_j|}\int_{Q_j}w^r\right)^{1/r}\le \int_0^1(w\chi_{Q_j})^*(\la|Q_j|)\frac{d\la}{\la^{1-1/r}}$$
and taking $r>1$ such that $1-1/r<\d$, we obtain the $SC_p$ condition.
\end{proof}

\section{Proof of Theorem \ref{weaksp}}
As we have mentioned in the Introduction, the proof of Theorem \ref{weaksp} is an adaptation of the method from \cite{DLR}.

\begin{proof}[Proof of Theorem \ref{weaksp}]
Denote
$$E=\{x:A_{\mathcal S}f(x)>2, Mf(x)\le 1/4\}.$$
Then, by homogeneity and by Chebyshev's inequality, it suffices to show that
\begin{equation}\label{suftos}
w(E)\lesssim\int_{{\mathbb R}^n}(Mf)^pw.
\end{equation}

By the standard limiting argument, one can assume that the family ${\mathcal S}$ is finite. Then $w(E)<\infty$.
For $k\in {\mathbb N}$ denote
$$F_k=\{Q\in {\mathcal S}: 4^{-k-1}<|f|_Q\le 4^{-k}\}.$$
Then, for $x\in E$,
$$A_{\mathcal S}f(x)\le \sum_{k=1}^{\infty}\frac{1}{4^k}\sum_{Q\in F_k}\chi_Q.$$
Therefore, by Chebyshev's inequality,
\begin{equation}\label{intre}
w(E)\le \frac{1}{2}\sum_{k=1}^{\infty}\frac{1}{4^k}\sum_{Q\in F_k}w(E\cap Q).
\end{equation}

Write ${F}_{k}=\cup_{\nu=0}^N{F}_{k,\nu}$, where
${F}_{k,0}$ is the family of the maximal cubes in ${F}_{k}$
and ${F}_{k,\nu+1}$ is the family of the maximal cubes in ${F}_{k}\setminus\bigcup_{l=0}^{\nu}{F}_{k,l}$.

Denote $E_{Q}=Q\setminus\bigcup_{Q'\in{F}_{k,\nu+1}}Q'$
for each $Q\in{F}_{k,\nu}$. Then the sets $E_Q$ are pairwise disjoint for $Q\in {F}_k$.

For $\nu\ge 0$ and $Q\in {F}_{k,\nu}$ denote
$$A_k(Q)=\bigcup_{Q'\in{F}_{k,\nu+2^k},Q'\subset Q}Q'$$
(if ${F}_{k,\nu+2^k}=\emptyset$, then set $A_k(Q)=\emptyset$).
Observe that
$$Q\setminus A_{k}(Q)=\bigcup_{l=0}^{2^k-1}\bigcup_{Q'\in{F}_{k,\nu+l},Q'\subseteq Q}E_{Q'}.$$

Using that the sets $E_Q$ are disjoint, we obtain
\begin{eqnarray*}
\sum_{Q\in {F}_k}
w\big(E\cap (Q\setminus A_{k}(Q))\big)&\le&
\sum_{\nu=0}^{N}\sum_{Q\in\mathcal{F}_{k,\nu}}\sum_{l=0}^{2^k-1}\sum_{\stackrel{{\scriptstyle Q'\in{F}_{k,\nu+l}}}{Q'\subseteq Q}}w(E\cap E_{Q'})\nonumber\\
&\le& 2^k\sum_{Q\in {F}_k}w(E\cap E_Q)\le 2^{k}w(E).
\end{eqnarray*}
From this and from (\ref{intre}),
$$w(E)\le \frac{1}{2}w(E)+\frac{1}{2}\sum_{k=1}^{\infty}\frac{1}{4^k}\sum_{Q\in F_k}w(A_k(Q)),$$
and hence
\begin{equation}\label{ns}
w(E)\le \sum_{k=1}^{\infty}\frac{1}{4^k}\sum_{Q\in F_k}w(A_k(Q)).
\end{equation}

By the $\eta$-sparseness, $|A_k(Q)|\le (1-\eta)^{2^k}|Q|$. Take $r>1$ as in the $SC_p$ condition. By H\"older's inequality,
$$w(A_k(Q))\le (1-\eta)^{2^k/r'}\left(\frac{1}{|Q|}\int_Qw^r\right)^{1/r}|Q|.$$
This along with (\ref{ns}) implies
\begin{equation}\label{nss}
w(E)\le \sum_{k=1}^{\infty}\frac{(1-\eta)^{2^k/r'}}{4^k}\sum_{Q\in F_k}\left(\frac{1}{|Q|}\int_Qw^r\right)^{1/r}|Q|.
\end{equation}

Let $Q_j$ be the maximal cubes of $F_k$. Then setting $F_k(Q_j)=\{Q\in F_k:Q\subseteq Q_j\}$, we can write $F_k=\cup_jF_k(Q_j)$.
Therefore,
$$\sum_{Q\in F_k}\left(\frac{1}{|Q|}\int_Qw^r\right)^{1/r}|Q|=\sum_j\sum_{Q\in F_k(Q_j)}\left(\frac{1}{|Q|}\int_Qw^r\right)^{1/r}|Q|.$$

By the sparseness and the well known fact that for $0<\d<1$,
$$\int_Q(M(f\chi_Q))^{\d}\le C_{\d,n}\left(\frac{1}{|Q|}\int_Q|f|\right)^{\d}|Q|,$$
we obtain
\begin{eqnarray*}
&&\sum_{Q\in F_k(Q_j)}\left(\frac{1}{|Q|}\int_Qw^r\right)^{1/r}|Q|\le \frac{1}{\eta}\sum_{Q\in F_k(Q_j)}\left(\frac{1}{|Q|}\int_Qw^r\right)^{1/r}|E_Q|\\
&&\le\frac{1}{\eta}\int_{Q_j}M(w^r\chi_{Q_j})^{1/r}\le C\left(\frac{1}{|Q_j|}\int_{Q_j}w^r\right)^{1/r}|Q_j|.
\end{eqnarray*}

Combining this with (\ref{nss}) and applying the $SC_p$ condition along with Lemma \ref{Mchi}, we obtain
\begin{eqnarray*}
w(E)&\lesssim& \sum_{k=1}^{\infty}\frac{(1-\eta)^{2^k/r'}}{4^k}\sum_j\left(\frac{1}{|Q_j|}\int_{Q_j}w^r\right)^{1/r}|Q_j|\\
&\lesssim& \sum_{k=1}^{\infty}\frac{(1-\eta)^{2^k/r'}}{4^k}\int_{{\mathbb R}^n}(M_{\chi_{\cup_jQ_j}})^pw\\
&\lesssim& \sum_{k=1}^{\infty}\frac{(1-\eta)^{2^k/r'}}{4^k}\int_{{\mathbb R}^n}(M\chi_{\{Mf>4^{-k-1}\}})^pw\\
&\lesssim& \sum_{k=1}^{\infty}(1-\eta)^{2^k/r'}4^{(p-1)k}\int_{{\mathbb R}^n}(Mf)^pw\lesssim \int_{{\mathbb R}^n}(Mf)^pw.
\end{eqnarray*}
This proves (\ref{suftos}), and therefore, the theorem is proved.
\end{proof}

\section{Proof of Theorems \ref{singchar} and \ref{FefS}}
In the necessity part  of Theorem \ref{singchar} we will use the notion of the grand maximal truncated operator $M_T$ defined in \cite{L4}
by
$$M_Tf(x)=\sup_{Q\ni x}\|T(f\chi_{{\mathbb R}^n\setminus 3Q})\|_{L^{\infty}(Q)}.$$

It was shown in \cite{L4} that for any Calder\'on-Zygmund operator with Dini-continuous kernel,
$$M_Tf(x)\lesssim T^{\star}f(x)+Mf(x).$$
Therefore, the assumption
\begin{equation}\label{assump}
\|R_k^{\star}f\|_{L^{p,\infty}(w)}\lesssim \|Mf\|_{L^p(w)}\quad(k=1,\dots,n)
\end{equation}
implies
\begin{equation}\label{imp1}
\|M_{R_k}f\|_{L^{p,\infty}(w)}\lesssim \|Mf\|_{L^p(w)}\quad(k=1,\dots,n).
\end{equation}
Also, (\ref{assump}) trivially implies that
\begin{equation}\label{imp2}
\|R_kf\|_{L^{p,\infty}(w)}\lesssim \|Mf\|_{L^p(w)}\quad(k=1,\dots,n).
\end{equation}

\begin{proof}[Proof of Theorem \ref{singchar}]
The sufficiency of the condition $w\in SC_p$ is an immediate combination of Theorems \ref{weaksp} and \ref{point}.

Let us turn to the necessity of $w\in SC_p$.
We will show that (\ref{imp1}) along with (\ref{imp2}) implies
condition (ii) of Theorem \ref{eqv} with $R=3$.

Take any sequence of cubes $\{Q_j\}$, which is 3-separated,
and let us show that
\begin{equation}\label{letussh}
\sum_j\|w\|_{L\log L,Q_j}|Q_j|\lesssim \int_{{\mathbb R}^n}(M\chi_{\cup_jQ_j})^pw.
\end{equation}

By Lemma \ref{str},
$$
\|w\|_{L\log L,Q_j}|Q_j|\lesssim \sum_{k=1}^n\int_{3Q_j}|R_k(w\chi_{Q_j})|dx+\int_{Q_j}wdx.
$$
Therefore, in order to prove (\ref{letussh}), it suffices to show that for every $k=1,\dots,n$,
\begin{equation}\label{sufff}
\sum_j\int_{3Q_j}|R_k(w\chi_{Q_j})|dx\lesssim \int_{{\mathbb R}^n}(M\chi_{\cup_jQ_j})^pw.
\end{equation}

Denote
$$\psi_j=\text{sign}\,R_k(w\chi_{Q_j})\chi_{3Q_j}\quad\text{and}\quad \psi=\sum_j\psi_j.$$
Then
\begin{eqnarray*}
&&\sum_j\int_{3Q_j}|R_k(w\chi_{Q_j})|dx=-\sum_j\int_{Q_j}R_k(\psi_j)wdx\\
&&=\int_{\cup_jQ_j}R_k(-\psi)wdx+\sum_j\int_{Q_j}R_k(\psi-\psi_j)wdx.
\end{eqnarray*}
Since
\begin{eqnarray*}
\sum_j\int_{Q_j}R_k(\psi-\psi_j)wdx&\le& \sum_j\int_{Q_j}|R_k(\psi\chi_{{\mathbb R}^n\setminus 3Q_j})|wdx\\
&\le& \int_{\cup_jQ_j}M_{R_k}(\psi)wdx,
\end{eqnarray*}
we obtain
\begin{equation}\label{obt}
\sum_j\int_{3Q_j}|R_k(w\chi_{Q_j})|dx\le\int_{\cup_jQ_j}\big(|R_k(\psi)|+M_{R_k}(\psi)\big)wdx.
\end{equation}

Applying (\ref{imp1}) and (\ref{imp2}) yields
\begin{eqnarray*}
&&\int_{\cup_jQ_j}\big(|R_k(\psi)|+M_{R_k}(\psi)\big)wdx\\
&&\le \|\chi_{\cup_jQ_j}\|_{L^{p',1}(w)}\||R_k(\psi)|+M_{R_k}(\psi)\|_{L^{p,\infty}(w)}\\
&&\lesssim w(\cup_jQ_j)^{1/p'}\|M\chi_{\cup_j3Q_j}\|_{L^p(w)}.
\end{eqnarray*}

We have already seen in the proof of the implication $({\rm{v}})\Rightarrow {\rm (i)}$ of Theorem \ref{eqv} that
$$\|M\chi_{\cup_j3Q_j}\|_{L^p(w)}\lesssim \|M\chi_{\cup_jQ_j}\|_{L^p(w)}.$$
Therefore,
$$\int_{\cup_jQ_j}\big(|R_k(\psi)|+M_{R_k}(\psi)\big)wdx\lesssim \int_{{\mathbb R}^n}(M\chi_{\cup_jQ_j})^pw,$$
which, along with (\ref{obt}), proves (\ref{sufff}), and therefore, the theorem is proved.
\end{proof}

\begin{proof}[Proof of Theorem \ref{FefS}]
Suppose that $w\in SC_p$. Let $f\in S_0({\mathbb R}^n)$. Fix a dyadic lattice ${\mathscr D}$. By Theorem \ref{md} combined with (\ref{oscr}),
there exists a $\frac{1}{6}$-sparse family ${\mathcal S}\subset {\mathscr D}$ such that for a.e. $x\in {\mathbb R}^n$,
$$|f|\le 2\sum_{Q\in {\mathcal S}}\o_{2^{-n-2}}(f;Q)\chi_Q.$$
Since $\o_{2^{-n-2}}(f;Q)\le (M_{2^{-n-2}}^{\#}f)_Q$, we obtain that
$$|f|\le 2A_{\mathcal S}(M_{2^{-n-2}}^{\#}f).$$
Therefore, by Theorem \ref{weaksp} combined with the left-hand side of (\ref{two}),
\begin{eqnarray*}
\|f\|_{L^{p,\infty}(w)}&\le& 2\|A_{\mathcal S}(M_{2^{-n-2}}^{\#}f)\|_{L^{p,\infty}(w)}\\
&\lesssim& \|MM_{2^{-n-2}}^{\#}f\|_{L^p(w)}\lesssim \|f^{\#}\|_{L^p(w)},
\end{eqnarray*}
proving (\ref{FSW}).

Assume now that (\ref{FSW}) holds. Let $E$ be a bounded set of positive measure, and let $0<\la<1$. Set in (\ref{FSW})
$f=\log^+(\frac{1}{\la}M\chi_E).$

It is well known (see \cite{CR}) that $f\in BMO$ and $\|f\|_{BMO}\le C_n$.
Hence,
\begin{equation}\label{linf}
\|M_{1/2}^{\#}f\|_{L^{\infty}}\le 2\|f\|_{BMO}\le 2C_n.
\end{equation}
Also, $\text{supp}\,(f)\subset \{x:M\chi_E\ge\la\}.$ Since
$$\text{supp}\,(M_{1/2}^{\#}f)\subset \{M\chi_{{\text{supp}}\,(f)}\ge 1/2\},$$
by Lemma \ref{Mchi} we obtain
$$\text{supp}\,(M_{1/2}^{\#}f)\subset \{M\chi_E\ge \la/2\cdot 9^n\}.$$
Therefore, by the right-hand side of (\ref{two}) along with (\ref{linf}),
\begin{eqnarray*}
\|f^{\#}\|_{L^p(w)}&\lesssim& \|MM_{1/2}^{\#}f\|_{L^p(w)}\\
&\lesssim& \|M\chi_{\{M\chi_E\ge \la/2\cdot 9^n\}}\|_{L^p(w)}.
\end{eqnarray*}

On the other hand, since $f\ge \log(1/\la)\chi_E$, we obtain that
$$\log(1/\la)w(E)^{1/p}\le \|f\|_{L^{p,\infty}(w)},$$
which along with the previous estimate and (\ref{FSW}) implies
$$w(E)\le \frac{C}{(\log(1/\la))^p}\int_{{\mathbb R}^n}(M\chi_{\{M\chi_E\ge \la/2\cdot 9^n\}})^pw.$$

Thus, $w$ satisfies condition (iii) of Theorem \ref{eqv}, which proves that $w\in SC_p$.
\end{proof}

\section{Concluding remarks}
\begin{remark}\label{rem1}
The main results of this paper raise several interesting questions. The most natural one is the following.

\begin{que}\label{q1}
What is the relationship between the $SC_p$ and $C_p$ conditions? In particular, is it true that $SC_p=C_p$?
\end{que}

If $SC_p\not=C_p$, then we would obtain that Muckenhoupt's conjecture for (\ref{CF}) as well as its counterpart for (\ref{FS})
are not true. Probably in this case, the $SC_p$ condition would be the natural candidate for a necessary and sufficient condition for (\ref{CF}) and (\ref{FS}).

Since the $C_{p+\e}$ condition for $p>1$ implies (\ref{FS}) (by Yabuta's result~\cite{Y}), by Theorem \ref{FefS} we obtain that
$$C_{p+\e}\Rightarrow SC_p\quad(p>1).$$
Therefore, thinking about a possible counterexample to the implication $C_p\Rightarrow SC_p$, a weight $w$ should be from the class
$C_p\setminus \cup_{q>p}C_q$. 
\end{remark}

\begin{remark}\label{rem2}
Observe that for the Hardy-Littlewood maximal operator $M$, there is an argument (see \cite[Cor. 1.3]{LO1} ) showing that
$$M:L^p(w)\to L^{p,\infty}(w)\Rightarrow M:L^p(w)\to L^p(w)\quad(p>1)$$
(without the use of the implication $M:L^p(w)\to L^{p,\infty}(w)\Rightarrow w\in A_p$).
Related to this, one can ask the following.

\begin{que}\label{q2}
Is it possible to deduce that the weak $L^p(w)$ Coifman-Fefferman inequality (\ref{CFW}) implies the strong Coifman-Fefferman inequality (\ref{CF}) (without appealing to any structural properties of $w$)?
\end{que}

The same question can be asked about the Fefferman-Stein inequality. If the answer to Question \ref{q2} is positive, then the $SC_p$
condition would be necessary and sufficient for (\ref{CF}) (at least with the maximal truncated Calder\'on-Zygmund operator $T^{\star}$).
\end{remark}

\begin{remark}\label{rem3}
As we have seen, the necessity part of the proof of Theorem \ref{singchar} is based essentially on the notion of the grand maximal truncated operator $M_T$, and this explains why Theorem \ref{singchar} is formulated for $T^{\star}$ instead of $T$. Thus, the following question appears naturally.

\begin{que}\label{q3}
Is it possible to deduce that the $SC_p$ condition is necessary for (\ref{CFW}) at least in the one-dimensional case for the Hilbert transform?
\end{que}

Since the $SC_p$ condition is necessary, in general, for the corresponding weak type estimate for the sparse operator, it would be very surprising if the answer to Question \ref{q3} is negative.
\end{remark}


\begin{thebibliography}{99}

\bibitem{Ca}
J. Canto,
{\it Quantitative $C_p$ estimates for Calder\'on-Zygmund operators}, preprint.
Available at https://arxiv.org/abs/1811.05209

\bibitem{CLRT}
J. Canto, K. Li, L. Roncal and O. Tapiola,
{\it $C_p$ estimates for rough homogeneous singular integrals and sparse forms},
preprint. Available at https://arxiv.org/abs/1909.08344

\bibitem{CLPR}
M. Cejas, K. Li, C. P\'erez and I.P. Rivera-R\'ios,
{\it Vector-valued operators, optimal weighted estimates and the $C_p$ condition},
preprint. Available at https://arxiv.org/abs/1712.05781

\bibitem{C}
R.R. Coifman, {\it Distribution function inequalities for singular integrals}, Proc. Nat. Acad. Sci. U.S.A. {\bf 69} (1972), 2838--2839.

\bibitem{CF}
R.R. Coifman and C. Fefferman, {\it Weighted norm inequalities for maximal functions and singular integrals}, Studia Math. {\bf 51} (1974), 241--250.

\bibitem{CR}
R.R. Coifman and R. Rochberg, {\it Another characterization of BMO}, Proc. Amer. Math. Soc. {\bf 79} (1980),
no. 2, 249--254.

\bibitem{CR1}
J.M. Conde-Alonso and G. Rey, {\it
A pointwise estimate for positive dyadic shifts and some applications}, Math. Ann.  {\bf 365} (2016), no. 3--4, 1111--1135.

\bibitem{DLR}
C. Domingo-Salazar, M. Lacey and G. Rey,
{\it Borderline weak-type estimates for singular integrals and square functions},
Bull. Lond. Math. Soc. {\bf 48} (2016), no. 1, 63--73.

\bibitem{FS}
C. Fefferman and E.M. Stein, {\it $H^p$ spaces of several variables}, Acta Math. {\bf 129} (1972), no. 3-4, 137--193.

\bibitem{G}
L. Grafakos, Modern Fourier analysis. Third edition. Graduate Texts in Mathematics, 250. Springer, New York, 2014.

\bibitem{JT}
B. Jawerth and A. Torchinsky, {\it Local sharp maximal functions},
J. Approx. Theory, {\bf 43} (1985), 231--270.

\bibitem{H}
T.P. Hyt\"onen, {\it
The $A_2$ theorem: remarks and complements}, Harmonic analysis and partial differential equations,
91--106, Contemp. Math., 612, Amer. Math. Soc., Providence, RI, 2014.

\bibitem{HRT}
T. Hyt\"onen, L. Roncal and O. Tapiola, {\it Quantitative weighted estimates for rough homogeneous singular integrals},
Israel J. Math., {\bf 218} (2017), no. 1, 133--164.


\bibitem{KM}
L. Kahanp\"a\"a and L. Mejlbro, {\it Some new results on the
Muckenhoupt conjecture concerning weighted norm inequalities
connecting the Hilbert transform with the maximal function},
Proceedings of the second Finnish-Polish summer school in complex
analysis (Jyv\"askyl\"a, 1983), 53--72, Bericht, 28, Univ.
Jyv\"askyl\"a, Jyv\"askyl\"a, 1984.

\bibitem{KK}
S. Kislyakov and N. Kruglyak, Extremal problems in interpolation theory, Whitney-Besicovitch coverings, and singular integrals.
Mathematical Monographs (New Series), 74. Birkhäuser/Springer Basel AG, Basel, 2013.

\bibitem{La}
M.T. Lacey, {\it An elementary proof of the $A_2$ bound}, Israel J. Math.  {\bf 217} (2017),  no. 1, 181--195.


\bibitem{L1}
A.K. Lerner, {\it On the John-Str\"omberg characterization of $BMO$
for nondoubling measures}, Real. Anal. Exchange, {\bf 28} (2003),
no. 2, 649--660.

\bibitem{L2}
A.K. Lerner, {\it Some remarks on the Fefferman-Stein inequality},
J. Anal. Math. {\bf 112} (2010), 329--349.

\bibitem{L3}
A.K. Lerner, {\it A pointwise estimate for the local sharp maximal function with applications to singular integrals},
Bull. Lond. Math. Soc. {\bf 42} (2010),  no. 5, 843--856.

\bibitem{L4}
A.K. Lerner, {\it On pointwise estimates involving sparse operators}, New York J. Math.  {\bf 22} (2016), 341--349.

\bibitem{LN}
A.K. Lerner and F. Nazarov,
{\it Intuitive dyadic calculus: The basics}, Expo. Math. {\bf 37} (2019), no. 3, 225--265.

\bibitem{LO1}
A.K. Lerner and S. Ombrosi, {\it A boundedness criterion for general maximal operator}, Publ. Mat., {\bf 54} (2010), no. 1, 53--71.

\bibitem{LO}
A.K. Lerner and S. Ombrosi, {\it Some remarks on the pointwise sparse domination}, J. Geom. Anal., to appear. Available at
https://arxiv.org/abs/1901.00195

\bibitem{M}
B. Muckenhoupt, {\it Norm inequalities relating the Hilbert
transform to the Hardy-Littlewood maximal function}, Functional
analysis and approximation (Oberwolfach, 1980), 219--231, Internat.
Ser. Numer. Math., 60, Birkh\"auser, Basel-Boston, Mass., 1981.

\bibitem{S2}
E.T. Sawyer, {\it Norm inequalities relating singular integrals and
the maximal function}, Studia Math., {\bf 75} (1983), 253--263.

\bibitem{St}
E.M. Stein, {\it Note on the class $L\log L$}, Studia Math. {\bf 32} (1969), 305--310.

\bibitem{St1}
E.M. Stein, Harmonic analysis: real-variable methods, orthogonality, and oscillatory integrals. With the assistance of Timothy S. Murphy. Princeton Mathematical Series, 43.
Monographs in Harmonic Analysis, III. Princeton University Press, Princeton, NJ, 1993.

\bibitem{W}
J.M. Wilson, Weighted Littlewood-Paley Theory and Exponential-Square Integrability, Lecture Notes in Math.,
vol. 1924, Springer-Verlag, 2008.

\bibitem{Y}
K. Yabuta, {\it Sharp maximal function and $C_p$ condition}, Arch.
Math. {\bf 55} (1990), no. 2, 151--155.


\end{thebibliography}
\end{document}